\documentclass[12pt]{article}

\usepackage{amsfonts,amssymb,graphicx}
\usepackage{amsmath}
\usepackage{color}
\usepackage[colorlinks=true, linkcolor=blue,citecolor=blue,urlcolor=blue]{hyperref}
\usepackage{hyperref}
\usepackage{amsmath, amscd,url}

\newtheorem{theorem}{Theorem}[section]
\newtheorem{lemma}[theorem]{Lemma}
\newtheorem{proposition}[theorem]{Proposition}
\newtheorem{corollary}[theorem]{Corollary}
\newtheorem{remark}{Remark}
\newtheorem{problem}{Problem}

\newtheorem{unnumber}{}

\newtheorem{prepf}[unnumber]{Proof}
\newenvironment{proof}{\prepf\rm}{\endprepf}

\usepackage{amssymb, amsmath, lineno}
\usepackage{graphicx}
\def\0{\Gamma_{\mathbb{Z}_{n},H}}
\def\c{\o{\Gamma}_{G,H}}
\def\1{\Gamma_{G,H}}
\def\2{\Gamma_{G}}
\def\3{\Gamma_{H}}
\def\4{\O_{|G|,|x|}}
\def\5{\O_{|H|,|x|}}
\def\a{\alpha}
\def\G{\Gamma}

\def\t{\text}
\def\sm{\setminus}
\def\se{\subseteq}
\def\s{\subset}

\def\O{\Omega}
\def\o{\overline}
\def\T{\theta}
\def\OT{\o{\T}}
\def\se{\subseteq}

\def\s{\subset}

\newcommand{\Aut}{\mathop{\mathrm{Aut}}}

\newcommand{\tc}{\mathop{\mathrm{TC}}}

\begin{document}
\title{Generalized non-coprime graphs of groups}
\author{S. Anukumar Kathirvel\\
{\small Department of Mathematics}\\
{\small Amrita College of Engineering and Technology}\\ 
{\small Kanyakumari 629901, Tamil Nadu, India}\\
{\small ORCID:0000-0002-1286-1444.}\\
Peter J. Cameron \\
{\small School of Mathematics and Statistics}\\
{\small University of St Andrews, North Haugh St Andrews}\\ 
{\small  Fife, KY16 9SS, U.K.}\\
{\small Email:~pjc20@st-andrews.ac.uk}\\ 
{\small  ORCID:0000-0003-3130-9505.} \\
and\\
T. Tamizh Chelvam\\
{\small Department of Mathematics}\\
{\small  Manonmaniam Sundaranar University}\\
{\small Tirunelveli 627012, Tamil Nadu, India}\\
{\small Email:~tamche59@gmail.com}\\
{\small  ORCID:0000-0002-1878-7847.}}
\date{August 2022}
\maketitle

 \begin{abstract}Let $G$  be a finite group with identity $e$ and \( H \neq \{e\}\) be a  subgroup of $G.$ The generalized non-coprime graph $\1$ of  \( G \) with respect to \(H\) is the simple undirected graph with \(G  \sm \{e \}\) as the vertex set and two distinct vertices \( a \) and \( b\) are adjacent if and only if $\gcd(|a|,|b|) \neq 1 $ and either $a \in H$  or $b \in H$, where $|a|$ is the order of $a\in G.$  In this paper, we study certain graph theoretical properties of generalized non-coprime graphs of finite groups, concentrating on cyclic groups. More specifically, we obtain necessary and sufficient conditions for the generalized non-coprime graph of a cyclic group to be in the class of stars, paths, cycles, triangle-free, complete bipartite, complete, unicycle, split, claw-free, chordal or perfect graphs. Then we show that widening the class of groups to all finite nilpotent groups gives us no new graphs, but we give as an example of contrasting behaviour the class of EPPO groups (those in which all elements have prime power order). We conclude with a connection to the Gruenberg--Kegel graph.
\end{abstract}
\section{Introduction}
Throughout this paper $G$ is a finite group with identity $e$. One can associate a graph to $G$ in many different ways, and there are several studies on graphs from groups. Several authors studied about graphs like Cayley graphs, power graphs, commuting graphs and others~\cite{CS,PRT,GR,MGA,RPTT19,RPTT20,SS,SW,TM}. In fact, the first ever construction of a graph from a group starts from the well known graph construction of Cayley graphs from finite groups~\cite{BS,DGMS,EP,GSL,KS,TTSM1,TTSM2,TTSR3,TTMS}. Since the order of an element in a group is one of the most basic parameters in group theory, Ma \textit{et~al.}~\cite{MWY} defined the coprime graph $\Gamma_G$ of $G$ using orders of elements in $G$. The coprime graph $\Gamma_G$ of $G$ is the simple undirected graph with vertex set $G$ two distinct vertices $x$ and $y$ are adjacent in $\Gamma_G$ if and only if $\gcd(|x|,|y|) = 1$. Ma \textit{et~al.}~\cite{MWY} obtained a necessary and sufficient condition for $\Gamma_G$ to have $\Aut(\Gamma_G)\cong\Aut(G)$ where $\Aut(\2)$ and $\Aut(G)$ are the automorphism groups of $\2$ and $G$ respectively. Dorbidi~\cite{GR} obtained a necessary and sufficient condition for $\Gamma_G$ to be planar. Further Selvakumar \textit{et~al.}~\cite{SS} classified all finite groups whose coprime graphs are toroidal or projective.   In \cite{MMAT}, the authors defined a new graph called the generalized non-coprime graph, and proved that the chromatic number and the clique number of the generalized coprime graph are equal for certain groups. 

The non-coprime graph $\Gamma_G$ of $G$ is the simple undirected graph with vertex set $G \sm\{e\}$ and two distinct vertices are adjacent whenever their orders are not relatively prime. Mansoori \textit{et~al.}~\cite{MET} obtained a necessary and sufficient condition for the non-coprime graph of a finite group to be planar or regular. Subsequently the notion of non-coprime graph was generalized with respect to a subgroup $H$ of $G$. In detail, Gholamreza \textit{et~al.}~\cite{GAA} defined the generalized non-coprime graph $\1$ of $G$ with respect to a subgroup $H$ of $G$ as the simple undirected graph with vertex set $G\sm\{e\}$ and  two distinct vertices are adjacent whenever their orders are not relatively prime and at least one of them is an element of $H$. In the same paper, the authors obtained a necessary and sufficient condition for the graph $\1$ to be planar or to contain a perfect matching.  

Let $\G=(V,E)$ be a graph. We say that $\G$ is connected if there is a path between any two distinct vertices of $\G$, otherwise $\G$ is disconnected. At the other extreme, we say that $\G$ is totally disconnected if no two vertices of $\G$ are adjacent. For two distinct vertices $x$ and $y$ of $\G$, the length of a shortest path from $x$ to $y$ is denoted by $d(x,y)$ if a path from $x$ to $y$ exists. Also we define $d(x,x)=0$ and $d(x,y)=\infty$ if there exists no path between distinct vertices $x$ and $y$. The {\it diameter} of $\G$ is defined as $diam(\G)=\sup\{d(x,y) : x,y\in V(\G)\}.$ The \textit{girth} of $\G$, denoted by $gr(\G)$, is the length of a smallest cycle in $\G,$ whereas $gr(\G)=\infty$ if $\G$ contains no cycles. The \textit {complement} $\overline{\G}$ of $\G$ is the graph whose vertex set is $V(\G)$ and such that for a pair $u,v$ of distinct vertices of $\G$, $uv$ is an edge of $\overline{\G}$ if and only if $uv$ is not an edge of $\G.$ Graphs $\G$ and $\G'$ are said to be \textit{isomorphic} to one another, written $\G\cong \G'$, if there exists a one-to-one correspondence $f:V (\G) \rightarrow V (\G')$ such that for each pair $x,y$ of vertices of $\G$, we have $xy \in E(\G)$ if and only if $f(x)f(y) \in E(\G')$. The graphs $K_n$, $K_{m,n},$ $P_n$ and $C_n$ denote respectively the complete graph on $n(\geq 1)$ vertices, the complete bipartite graph with vertex partition of subsets of  cardinality $m(\geq 1)$ and $n(\geq 1),$ the path with $n$ vertices and the cycle with $n$ vertices respectively. For a subset $H$ of the vertex set $V$ of a graph $\Gamma,$ $\langle H\rangle$ denotes the subgraph induced by $H.$ In fact $\langle H\rangle$ is a graph with vertex set $H$ and two distinct vertices in $H$ are if they are already adjacent in $\Gamma.$ For basic definitions in graph theory, we refer to Chartrand \textit{et~al.}~\cite{CHA06}. 

Throughout this paper, \(G\) is a finite group, and \(H\neq \{e\}\) a subgroup of \(G\). For an element  $x\in G$,  $|x|$ denotes the order of  $x$.  For a finite group $G$ of order $n=p_1^{\a_1}p_2^{\a_2} \ldots p_r^{\a_r}$, where $p_1,\ldots,p_r$ are distinct primes and $a_1,\ldots,a_r$ positive integers, and $x\in G,$ we let $\T_x=\{ p_i \in \{p_1,\ldots,p_r \}: p_i \mid |x|\} $ and $\OT_x= \{p_1,\ldots,p_r \} \sm \T_x .$ For a subgroup $H$ of $G,$ we take $\T_H=\{ p_i \in \{p_1,\ldots,p_r \}: p_i \mid |H|\} $ and $\OT_H= \{p_1,\ldots,p_r \} \sm \T_H .$  For a positive integer $n,$ $\phi(n)$ is the Euler function of $n$ and $ \tau(n) $ denotes the number of positive divisors of $n.$ For two integers $a$ and $b,$ $(a,b)$ denotes the greatest common divisor of $a$ and $b.$  For positive integers $n$ and  $h,$ we denote $\O_{n,h}=\{d : d\mid n,~ d>1  \text{~and~}~(d,h) > 1 \}$ and  $\o{\O}_{n,h}=\{d :d\mid n,~d>1\text{~and~}~(d,h) =1\}$. Further group-theoretic notation will be introduced as required. We caution the reader that, if $X$ is a subset of a group $G$, then $\langle X\rangle$ means the subgraph of $\Gamma_{G,H}$ induced on $X$ and not the subgroup generated by $X$.

The outline of this paper is as follows. In section 2, which is the heart of  the paper, we consider the case where $G$ is a cyclic group. We first obtain the degree of all vertices of the generalized non-coprime graph $\1$ of any finite cyclic group $G$ and a subgroup $H\neq \{e\}$ of $G.$ This in turn gives the maximum and minimum degrees of the graph $\1.$ Further, we provide a characterization for $\1$ to be connected and also prove that $\1$ is not Eulerian.  In section 3,  we provide  characterizations  for the graph $\1$ to be in the class of stars, paths, cycles, complete bipartite graphs,  complete graphs,  unicyclic graphs,  split graphs,  claw-free graphs and  chordal graphs. In the final part, we obtain a characterization for $\1$ to be a perfect graph. 

The choice of cyclic groups is not as special as it appears. In the third section, we widen the class to include all finite nilpotent groups, and find that no new graphs are obtained; every generalized non-coprime graph of a finite nilpotent group is isomorphic to the generalized non-coprime graph of a cyclic group. In the fourth section we treat a contrasting class, that of EPPO groups 
(groups where every element has prime power order). We completely describe the graphs that arise, and find that they are not particularly interesting. The final section links our results to the Gruenberg--Kegel graph of the group (a much smaller graph whose vertices are the prime divisors of $|G|$).

\section{Finite cyclic groups}
\subsection{Basics of the generalized non-coprime graph}
In this section, we obtain the degree of all vertices of the generalized non-coprime graph $\1$ for any finite cyclic group $G$ with respect to the subgroup $H\neq \{e\}.$  Hence we find the maximum and minimum degrees of $\1.$ Further, we obtain the characterization for the graph $\1$ to be connected. In the following lemma, we observe certain basic properties of the graph $\1$  and the same are useful in the proofs of subsequent results of this paper.

\begin{lemma}\label{1}	
Let \(G\) be a finite cyclic group of order $n=p_1^{\a_1}p_2^{\a_2} \ldots p_r^{\a_r}\geq 3$ where $p_i's$ are distinct primes and $\a_i's$ are non-negative integers for $1\leq i\leq r.$ Let $\1$ be the generalized non-coprime graph of $G$ with respect to a subgroup $\{e\}\neq H$ of $G.$ For $x\in G$, let $\T_x=\{ p_i \in \{p_1,\ldots,p_r \}: p_i \mid |x|\} $ Assume that $ x,~ y \in G \sm \{e \}$ and  $x \neq y.$ Then the following are true.
\begin{itemize}
\item [\rm (i)]  The vertices $x$ and $y$ are adjacent in $\1$ if and only if either $x \in H$ or $y \in H$ with $ \T_x \cap \T_y \neq \emptyset$. 
\item [\rm (ii)]  If $|x|=|y|$, then $\deg(x)=\deg(y)$ in $\1$.
\item [\rm (iii)]  Let $x,y\in G$ be such that $\T_x = \T_y$ and either $x,y \in H$ or $x,y \in G\sm H.$ Then $\deg(x)=\deg(y)$ in $\1$.
\item [\rm (iv)]  If $x \in H$ and $y \in G\sm H$ with $\T_x = \T_y,$ then $\deg(x) > \deg(y)$ in $\1$.
\end{itemize}
\end{lemma}  
\begin{proof}
(i)  The proof follows from the definition of the graph $\1.$
	
(ii)  For $x,y\in G, $ $|x|=|y|$ implies that $\theta_x=\theta_y$.  Suppose $z \in G\sm \{e,x,y \}$ and $z$ is adjacent to $x$ in $\1$.  Since $H$ is a subgroup of the cyclic group $G,$ we have $H$ is the only subgroup of order $|H|$ in $G.$ Hence either $x,y \in H$ or $x,y \in G\sm H.$ This, along with $|x|=|y|$ and $\gcd(|x|,|z|)>1$, gives that $z$ is adjacent to $y$. This implies that $\deg(x)=\deg(y)$ in $\1$.

(iii)  Suppose $ z \in  G \sm \{e,x,y\}$ such that $ z $ and $ x $ are adjacent in $\1$. By above (i), $\T_x \cap \T_z \neq \emptyset$ and so $ \T_z \cap \T_y \neq \emptyset$. Since either $ x,y \in H$ or $ x,y \in G \sm H, $ we have $ z $ and $ y $ are adjacent in $ \1$. This implies that $N(x)\subseteq N(y).$ By reverting the argument, we get that  $N(y)\subseteq N(x)$. Hence $\deg(x) = \deg(y)$ in ${\1}$. 
	
(iv)  Suppose $z$ is adjacent with $y$ in $\1$. This gives that  $\T_z\cap \T_y\neq \emptyset.$ Since $y \in G\sm H,$ $z \in H$ and $\T_z\cap \T_x\neq \emptyset$. This implies  that $z$ is adjacent to $x$. Thus, any element adjacent to $y$ is also adjacent to $x$ in $\1$.  Hence $\deg(x) \geq \deg(y)$ in $\1$. Suppose $z\in G \sm H$ and $\T_z \cap \T_x \neq \emptyset$. By (i), $x$ and $z$ are adjacent in $\1$. But $y,z \in G \sm H$ implies that $y$ and $z$ are not adjacent in $\1$. This gives that $\deg(x) > \deg(y)$ in $\1$.	
\end{proof}

In the following lemma, we obtain degrees of all the vertices of the generalized non-coprime graph $\1$ of a finite cyclic group $G$ with respect to a subgroup $H$ of $G.$
\begin{lemma}\label{2} Let \(G\) be a finite cyclic group of order $n=p_1^{\a_1}p_2^{\a_2} \ldots p_r^{\a_r}\geq 3$ where $p_i$ are distinct primes and $\a_i$ are non-negative integers for $1\leq i\leq r$. Let  $\1$ be the generalized non-coprime graph of $G$ with respect to a subgroup $\{e\}\neq H$. For positive integers $n$ and  $h$, let
\[\O_{n,h}=\{d : d \mid n,~ d>1  \text{~and~}~(d,h) > 1 \}.\]
Then, for  \(x \in G \sm \{e\} \),
\[\deg_{\1}(x) =\begin{cases}
\sum\limits_{d \in \O_{|G|,|x|}} \phi(d) - 1 \qquad & \text{~if~} x  \in H;\\
\sum\limits_{d \in \O_{|H|,|x|}} \phi(d) \qquad & \text{~if~}  x \in G \sm H.
\end{cases}\]
\end{lemma}
\begin{proof}
Suppose $x \in H$ and  $S=\{y \in G \sm \{e,x \} : |y| \in \4 \}$.  By definition,  $x$ and $ y $ are adjacent in $ \1$ if and only if  $ y \in S$ and so $\deg_{\1}(x) =|S|.$  Note that, if $m$ is a divisor of $|G|=n$, then $G$ contains exactly $ \phi(m) $ elements of order $m$.  This implies that 
\[|S|= \sum\limits_{d \in \4} \phi(d) - 1.\]

Suppose $ x\in G\setminus H$ and $S=\{y\in H : |y| \in \5 \}$.  By definition, $x$ and $ y $ are adjacent in $ \1$ for every $ y \in S$ and so $\deg_{\1}(x) = |S|$.  Note that $ H $ is a cyclic subgroup of $ G $ and so $ H $ contains exactly $ \phi(m) $ elements of order $ m$ where $ m$ divides $ |H|$. This implies that
\[|S|= \sum\limits_{d \in \5} \phi(d).\]
\end{proof}

In the following lemma, we find the maximum degree of the graph $\1$.
\begin{lemma}\label{3} Let \(G\) be a finite cyclic group of order $n=p_1^{\a_1}p_2^{\a_2}\ldots p_r^{\a_r}\geq 3$ where $p_i$ are distinct primes and $\a_i$ are non-negative integers for $1\leq i\leq r$. Let \(H \neq \{e\}\) be a subgroup of $G$ of order $h$. Let $\1$ be the generalized non-coprime graph of $G$ with respect to a subgroup $H$ of $G$ and let $\o{\O}_{n,h}=\{d :  d\mid n,d>1\text{~and~}~(d,h)=1\}.$  Then 
\[\Delta(\1) = \begin{cases}
|G|-2 & \t{~if~} p_i ~\mid ~ |H| \t{~for~every~} i(1 \leq i \leq r);\\
|G|-(\sum\limits_{d \in \o\O_{n,h}} \phi(d) + 1) & \t{~if~} p_i~\nmid~|H| \t{~for~some~}i (1 \leq i \leq r). 
\end{cases}\]
\end{lemma}
\begin{proof} 
Suppose $ p_i$ divides $|H|$ for every $i$ with $1 \leq i \leq r$. Let $x\in H$ be such that $ |H|=|x|$. Since  $p_i$ divides $|x| $ for every $i$, $x$ is adjacent to every element in $G \sm \{e,x\} $ in $ \1 $ and hence $\Delta(\1) = |G| - 2$.
	
Suppose $p_i \nmid |H|$ for some $i$ with $1 \leq i \leq r$. From  this, we get that $ \o{\O}_{n,h}=\{d : d\mid n, d>1~\text{~and~} (d,h)=1\}  \neq\emptyset$.
Let $ S=\{x \in G: |x| \in \o{\O}_{n,h}\}$. Then $S \s G \sm H$ with $|S|= \sum\limits_{d \in \o{\O}_{n,h}} \phi(d) $ and no non-identity element in $H$ is adjacent to any element in $S$ in $ \1 $ and i\emph{vice~versa}. By the definition of the graph $\1$, every element in $S$ is an isolated vertex in $\1$. This implies that
\[\Delta(\1) \leq |G| - \sum\limits_{d \in \o{\O}_{n,h}} \phi(d).\]
Let $x\in H$ be such that $|x|=|H|$ and let $T= G\sm  (S\cup \{e\}).$ Then $x$ is adjacent to every other element  in $T$ in $ \1 $ and so $\Delta(\1) = |G|-(\sum\limits_{d \in \o{\O}_{n,h}} \phi(d) + 1)$.
\end{proof}	

Aghababaei \textit{et~al.}~\cite[Theorem~3.5]{GAA} proved that, if $p_i \mid |H| \t{~for~every~}i(1 \leq i \leq r),$ then $\1$ is connected.  Further, it is proved that  $\1$ is not Eulerian if $G$ is a non-trivial $p$-group and $H$ is a subgroup of $G$ \cite[Theorem~3.11]{GAA}. In the following theorem, we extend the above results to a general finite cyclic group.  In fact, we show that $\1$ is not Eulerian for any finite cyclic group $G$ and a subgroup $H$ of $G$.
 
\begin{theorem} \label{4}
Let \(G\) be a finite cyclic group of order $n=p_1^{\a_1}p_2^{\a_2}\ldots p_r^{\a_r}\geq 3$ where $p_i$ are distinct primes and $\a_i$ are non-negative integers for $1\leq i\leq r$. Let $\1$ be the generalized non-coprime graph of $G$ with respect to a subgroup $\{e\}\neq H$ of $G.$ Then the following hold:
\begin{enumerate} 
\item [\rm (i)]   $\1$ is connected if and only if $p_i \mid |H|$ for every $i$
with $1 \leq i \leq r$;
\item [\rm (ii)]  $\1$ is not Eulerian.
\end{enumerate} 	
\end{theorem}
\begin{proof} (i) Assume that $p_i\mid |H|$ for every $i(1\leq i\leq r)$. By \cite[Theorem 3.5(1)]{GAA}, $\1$ is connected. Conversely, assume that $\1$ is connected. Suppose $p_i\nmid |H|$ for some $i(1\leq i\leq r)$. As observed in the proof of Lemma~\ref{3}, $\1$ contains at least one isolated vertex and so $\1$ is not connected, which is a contradiction.

(ii)  In view of \cite[Theorem~3.11]{GAA}  and Part (i) above, it is enough to prove that $\1$ is not Eulerian when  $p_i \mid |H|$ for every $i(1\leq i \leq r)$ and $r \geq 2$.

\textbf{Case 1.} Suppose $H\neq \{e\}$ is a proper subgroup of $G$.

\textbf{Case 1.1.}  Suppose $|G|$ is even. Here  $p_i \mid |H|$ for every $i$ with $1\leq i \leq r$ and $r \geq 2$. Hence $|H|$ is even. Let $x \in G$ be such that $|x|=|G|$. Then $x \in G \sm H$ and $x$ is adjacent with vertices in $H \sm \{e\}$ only.  Since $|H \sm \{ e \}|$ is odd, $\deg(x)$ is odd in $\1.$ Hence $\1$ is not Eulerian.

\textbf{Case 1.2.}  Suppose $|G|$ is odd. Then $2 \nmid |G|$. Let $x \in H$ such that $|x|=|H|.$ Then $x$ is  adjacent to every other element in  $G \sm \{e,x\}.$ Since  $|G\sm \{e,x\}|$ is odd, $\deg(x)$ is odd in $\1$ and so  $\1$ is not Eulerian.

\textbf{Case 2.} Suppose $H = G$.

When $r=1$, $\1$ is not Eulerian by \cite[Theorem~3.11]{GAA}. 

When $r\geq 2$, $G$ is not a $p$-group. Therefore an odd prime $p_2$ divides $|G|$.  Choose an element $x \in G$ such that $|x|=p_2$. If $y$ is adjacent to $x$ in $\1$, then $\phi(|y|)$ is even. By Lemma~\ref{2}, $\deg(x)$ is odd in $\1$ and so  $\1$ is not Eulerian.
\end{proof}

In the following lemma, we find the minimum degree $\delta(\1)$ of  the generalized non-coprime graph $\1$.

\begin{lemma}\label{5}	Let \(G\) be a finite cyclic group of order $n=p_1^{\a_1}p_2^{\a_2}\ldots p_r^{\a_r}\geq 3$ where $p_i$ are distinct primes and $\a_i$ are non-negative integers for $1\leq i\leq r$.  Let $\1$ be the generalized non-coprime graph of $G$ with respect to a subgroup $\{e\}\neq H$ of $G.$ Then the following are true:
\begin{enumerate}
\item [\rm (i)] If $G$ is a $p$-group, then $\delta(\1)=\begin{cases}
|G|-2 \qquad \text{~if~} H=G;\\
|H|-1 \qquad \text{~if~} H \subset G.
\end{cases}$
\item [\rm (ii)] If $G$ is not a $p$-group and $\1$ is connected, then  
$\delta(\1)= \\ \min \{\deg_{\1}(x_i) : x_i \in G $ and $|x_i|=p_i^{\a_i} \t{~for~some~} i(1 \leq i \leq r)  \};$
\item [\rm (iii)] If $\1$ is disconnected, then $\delta(\1)=0$.
\end{enumerate}
\end{lemma}

\begin{proof} (i) Assume that $G$ is a $p$-group.

Suppose $H=G$. Then $\1\cong K_{|G|-1}$, and so $\delta(\1)= |G|-2$. 

Suppose $H \subset G$.  Since $G$ is a $p$-group, $x\in H\sm \{e\}$ is adjacent with every element in $y\in G\sm H$ and vice-versa. Note that $|H|-1<|G\setminus H|$. This implies that $\deg(x)\geq |H|-1$ for all $x \in G$. Since $\langle G \setminus H\rangle$ is a totally disconnected subgraph of $\1$, $\delta(x)=|H|-1$ in $\1$.
 
(ii) Assume that $G$ is not a $p$-group and $\1$ is connected. By Theorem~\ref{4}(i), $p_i$ divides $|H|$ for every $i$ with $1 \leq i \leq r$. Let $S = \{x_i\in G  : |x_i|=p_i^{\a_i} ~\text{for~some}~i(1\leq i\leq r)\}$. Since $G$ is not a $p$-group, we have $S \neq \emptyset$.  If $y \in G \sm S$ and $p_i \mid |y|,$ then there exists an element $x_{p_i} \in S$ such that $\deg(y) \geq\deg(x_{p_i})$ in ${\1}$. This implies that $\delta(\1)=\min \{\deg_{\1}(x_i) : x_i \in G  $ and $|x_i|=p_i^{\a_i} \t{~for~some~} i~(1 \leq i \leq r)\}$.
 
(iii) Assume that $\1$ is disconnected. By Theorem~\ref{4}(i), $p_i\nmid |H|$  for~some~$i$ $(1 \leq i \leq r)$. As seen in the proof of Lemma~\ref{3}, $\1$ contains at least one isolated vertex and so $ \delta(\1)=0$.
\end{proof}	

\subsection{Characterizations of generalized non-coprime graphs}
In this section, we characterize all finite cyclic groups whose generalized non-coprime graph is in the class of star, path, cycle or triangle-free graphs.

\begin{lemma}\label{11.1}	
Let \(G\) be a finite cyclic group of order $n\geq 3$ and let $\1$ be the generalized non-coprime graph of $G$ with respect to a subgroup $\{e\}\neq H$ of $G$. 
Then the following hold good.
\begin{enumerate}
\item [\rm (i)] If $p$ is a prime number, $p\mid n$ and $H$ is a $p$-group, then $\langle H\setminus \{e\}\rangle$ is a clique in $\1$ and $\langle G\setminus H\rangle$ is an independent of $\1$;
\item[\rm (ii)] If $G$ is a $p$-group and $H\neq G,$ then $ \1$ contains a $K_{|H|-1,|G\setminus H|}$ as a subgraph. In particular, if $G$ is a $2$-group and $H \cong \mathbb{Z}_2$, then $ \1\cong K_{1,|G|-2};$
\item [\rm (iii)]  $\1$ is the  complete graph $K_{|G|-1}$ if and only if $G$ is a $p$-group and $H=G$. 
\end{enumerate}		 
\end{lemma}

\begin{proof} The proof of (i) is trivial.

(ii) Each element $x \in H\setminus \{e\}$ is adjacent to every element from $G\setminus H$ in $\1.$ Hence $\1 $ contains a $K_{|H|-1,|G\setminus H|}$ as a subgraph. The  part of the statement directly follows from these two conclusions. 

(iii) Assume that  $\1$ is complete. Suppose $H$ is not a $p$-group. Then there exist distinct primes $p_i$ and $p_j$ such that $p_ip_j \mid |H|$ for $i \neq j.$ Choose $x,y \in G$ such that $|x|=p_i$ and $|y|=p_j.$ Then $x$ and $y$ are not adjacent in $\1,$ which is a contradiction. Suppose $G$ is not a $p$-group and $H$ is a $p$-group. Then $\1$ is disconnected, which is a contradiction. Suppose $G$ is a $p$-group and $H\subset G.$ Then $\phi(|G|) \geq 2$ and so $|G\sm H|\geq 2.$ Since $\langle G\sm H\rangle$ is a totally disconnected subgraph in $\1,$ $\1$ is not complete, which is a contradiction.  Converse part follows from \cite[Theorem~3.4]{GAA}.
\end{proof}	

\begin{theorem}\label{11}	
Let \(G\) be a finite cyclic group of order $n\geq 3$ and let $\1$ be the generalized non-coprime graph of $G$ with respect to a subgroup $\{e\}\neq H$ of $G$. 
Then the following hold:
\begin{enumerate}
\item [\rm (i)] $\1$ is a star graph  if and only if either $G$ is a $2$-group and $H \cong \mathbb{Z}_2$ or $G=H=\mathbb{Z}_3.$
\item [\rm  (ii)] $\1$ is a path if and only if either $G=H=\mathbb{Z}_3$ or
$G\cong \mathbb{Z}_4$ and $H \cong \mathbb{Z}_2$.
\item [\rm  (iii)]  $\1$ is a cycle if and only if $G\cong$ $\mathbb{Z}_4$ and $H=G$.		
\item [\rm (iv)]  $\1$ is triangle-free if and only if either $G$ is a $2$-group and $H\cong \mathbb Z_2$ or  $H=G=\mathbb Z_3$.
\item [\rm  (v)]  $\1$ is complete bipartite if and only if either $G$ is a $2$-group and  $H\cong \mathbb{Z}_2$ or $H=G=\mathbb Z_3$.
\end{enumerate}		 
\end{theorem}
\begin{proof} (i)  Assume that $G$ is a $2$-group and $H \cong \mathbb Z_2$. Since $|G|\geq 3,$ $H\neq G$. By Lemma~\ref{11.1}(ii),  $\1=K_{1,|G|-2}$ and so it is a star graph.  If $G=H=\mathbb Z_3$, then $\1=K_{1,1},$ again a star graph.

Conversely, assume that $\1$ is a star graph. Suppose $G$ is not a $p$-group. Then $\phi(|G|)\geq 2$.

If $H=G,$ then choose $x$ and $y$ in $G$ such that $|x|=|y|=|G|$ and $x\neq y$. 
Then $x,y$ are adjacent to all other vertices in $\1$ and so $\deg(x)=\deg(y)=|G|-2 \geq 2$; so $\1$ is not a star graph. 

Suppose $H\neq G$ and $H$ is not a $2$-group. Then $\phi(|H|)\geq 2$ and choose $x\neq y \in H$ such that $|x|=|y|=|H|$. Let $u\neq v \in G$ such that $|u|=|v|=|G|$. Note that $u,v \in G \setminus H$. Both $u$ and $v$ are adjacent to $x,y$ in $\1$ and \emph{vice~versa}. This implies that $\1$ contains at least two distinct vertices of degree at least two; so $\1$ is not a star graph.

Suppose $H$ is a $2$-group. Then $\1$ is disconnected, so not a star graph.   

Thus $\1$ is not a star graph whenever $|G|$ is not a $p$-group. Hence $G$ is a $p$-group. When $p\geq 5$ by the above arguments, $\1$ is not a star graph. Hence $p$ is either $2$ or $3$. 
 
Let $p=3.$ Then the only possibility for $H$ is $G$. If $|G|\geq 4,$ then by Lemma~\ref{11.1}(iii), $\1$ is the complete graph $K_{|G|-1}$ and so $\1$ is not a star graph, which is a contradiction. Hence $|G|=3$ and $H=G$.

Let $p=2$. In this case $|G|\geq 4$. If $|H|\geq 4$, then by Lemma~\ref{11.1}(i), $\langle H\rangle$ is a clique and is a subgraph of $\1$,  which is a contradiction. Hence $H \cong \mathbb Z_2$. 
 
(ii) If $H=G=\mathbb Z_3$, then $\1=P_2$. If $G\cong \mathbb{Z}_4$ and  $|H|=2$, then $\1=P_3$.

Conversely, assume that $\1$ is a path. If $H=G$, then $G$ is a $p$-group and $|G|\geq 4$, then $\1$ is a complete graph with at least three vertices, which is a contradiction. If $\mathbb Z_2 \ncong H\ncong G $ and $G$ is a $p$-group, then $|H|\geq 3$ with two distinct elements $x,y \in H$ such that $|x|=|y|$ and $|G\sm H|\geq 2$. Note that $w \in G\sm H$ is adjacent to $x,y$. This gives that $\1$ contains a triangle, a contradiction. If $G$ is not a $p$-group and $H$ is a $p$-group, then $\1$ is disconnected, a contradiction. If $G$ is not a $p$-group and $H$ is not a $p$-group, then one can choose distinct elements $x,y,z \in H$ such that $|x|=|y|=p_j$ and $|z|=p_ip_j$, where $p_i$ and $p_j$ are distinct primes and $i \neq j$ with $j=2$. Then $\langle x,y,z\rangle=C_3$ in $\1$, which is a contradiction. Hence $\1$ is path if and only if $H=G=\mathbb Z_3$ or $|G|=4$ with $|H|=2$.
	
(iii) Assume that $G\cong$ $\mathbb Z_4$ with $H=G$. By \cite[Corollary~1]{MET}, $\1$ is a cycle. Conversely, assume that $\1$ is a cycle. Suppose $H=G$ with $|G|\neq 4.$ By \cite[Corollary~1]{MET}, $\1$ is not a cycle, which is a contradiction. If $H \neq G,$ then $|H\sm \{e\}|<|G\sm H|$ and $\langle G\sm H\rangle$ is a totally disconnected subgraph of $\1.$ Hence $\1$ is not a cycle, which is a contradiction. Hence $\1$ is a cycle if and only if $H=G= \mathbb Z_4$.  
	
(iv) The forward implication follows from part (i) above; the converse from the arguments in the proof of part (ii).
 	
(v) The proof follows from parts (i) and (iv).
\end{proof}
	  
A graph $\G$ is said to be \textit{unicyclic} if $\G$ contains exactly one cycle.	 In the following theorem, we obtain a characterization for the graph $\1$ to be unicylic.

\begin{theorem} \label{12.1} Let \(G\) be a finite cyclic group of order $n\geq  3$ and \(H\neq \{e\}\) be a subgroup of $G$. Then the generalized non-coprime graph $\1$ is unicyclic if and only if $H=G$ with $G\cong\mathbb Z_4$.		\end{theorem}
\begin{proof}		
Assume that $H=G$ with $G\cong \mathbb Z_4.$ Then $\1=C_3$ and so $\1$ is unicyclic. 

Conversely, assume that $\1$ is unicyclic.

Suppose $ H \neq G$ and  $H$ is a $p$-group. If $H \cong \mathbb Z_2$, then $\1$ is a star graph and so $\1$ has no cycle, a contradiction. Hence $|H|\geq 3$. Then there exist distinct elements $u,v\in H \setminus \{e\}$. Since $\phi(n) \geq 2,$ one can choose distinct elements $x,y \in G$ such that $|x|=|y|=|G|$. Then $x,y \in G \sm H$. Since both $x$ and $y$ are adjacent to any element from $H\sm \{e\}$ in $\1.$ From this, we have two cycles $\langle \{x,u,v\}\rangle$ and $\langle \{y,u,v\}\rangle>$ of length $3$ in $\1,$ which is a contradiction.

Suppose $ H \neq G$ and $H$ is not a $p$-group. Then there exist two distinct primes $p_i$ and $p_j$ such that $p_i, p_j \mid |H|$. Let $u,v \in H$ such that $|u|=p_i$ and $|v|=p_ip_j$. Since $H \neq G,$ one can choose distinct elements $x,y \in G$ such that $|x|=|y|=|G|$.  Then $\langle\{x,u,v\}\rangle$ and $\langle\{y,u,v\}\rangle$ are two different cycles of length $3$ in $\1,$ which is a contradiction. Hence $H=G$.  

If $|G|=3,$ then $\1$ is a path, a contradiction. If $|G|>4$ and $G$ is a $p$-group, then by Lemma~\ref{11.1}(iii), $\1$ is a complete graph of order greater than or equal to 4, again a contradiction. If $|G|>4$ and $G$ is not a $p$-group, then there exist two distinct primes divisors $p_i$ and $p_j$ for $|G|$. Choose distinct elements $u,v,x,y \in G$ such that $|u|=|v|=p_j$ and $|x|=|y|=p_ip_j.$ Then $\langle\{x,u,v\}\rangle$ and $\langle\{y,u,v\}\rangle$ are two different cycles of length $3$ in $\1,$ which is a contradiction. Hence $\1$ is unicyclic if and only if $H=G= \mathbb Z_4.$ 	
\end{proof}

A graph $\G$ is a {\it split graph} if $V(\G)$ can be partitioned into a clique and an independent set. A well-known characterization of split graphs is given below.

\begin{theorem}{\normalfont (\cite[Theorem 6.3]{FH77})}\label{thm4} A connected graph $\G$ is a split graph if and only if t$\G$ contains no induced subgraph isomorphic to $2K_2$, $C_4$ or $C_5$. 
\end{theorem}

In the following theorem, we obtain a characterization for the graph $\1$ to be a split graph.
\begin{theorem}\label{9} Let \(G\) be a finite cyclic group of order $n\geq 3$ and \(H\neq \{e\}\) be a subgroup of $G.$ Then the generalized non-coprime graph $\1$ is a split graph if and only if either of the following true.
\begin{enumerate}
\item[\rm (i)] There exists a prime number $p$ such that $p\mid |G|$ and $H$ is a $p$-group;
\item[\rm (ii)]  $H = G = \mathbb Z_6$.	
\end{enumerate}
\end{theorem}

\begin{proof} Suppose (i) is true. By Lemma~\ref{11.1}(i), $\langle H \sm \{e\}\rangle=K_{|H|-1}$ and $\langle G \sm H\rangle=\o{K}_{|G \sm H|}$ in $\1.$ From this, $\1$ is a split graph.

(ii) Assume that $H=G=\mathbb{Z}_6.$ If $S= \{1,2,4,5 \} \s \mathbb Z_6=H,$ then $\langle S\rangle=K_4$ and the vertex $3$ is adjacent to only $1$ and $5$ in $\1.$ Hence $\1$ is a split graph. 
	
Conversely, assume that $\1$ is a split graph. Assume that $G$ is not a $p$-group. Then there exists $r$ distinct prime divisors $p_i$ for $|G|$ with $r\geq 2$.

Suppose  that $H \neq G$.

Suppose that $p_i \mid |H|$ for every $i$ with $1 \leq i \leq r$. By Theorem~\ref{4}(i), $\1$ is connected. Let $x,y \in H$ such that $|x|=p_1$ and $|y|=p_2$. Since $\phi(|G|) \geq 2$, choose $u,v \in G$ such that $|u|=|v|=|G|$ with $u \neq v$. Then $\langle\{u,v,x,y \}\rangle=C_4$ in $\1$ and by Theorem~\ref{thm4}, $\1$ is not a split graph, which is a contradiction.

Suppose that $p_i \mid |H|$ and $p_j \nmid |H|$ for $1 \leq i,j\leq r$ and $i\neq j.$ By Theorem~\ref{4}(i), $\1$ is disconnected. As proved in Lemma~\ref{3}, $\1$ contains at least one isolated vertex. Let $S$ be the set of all isolated vertices in $\1$ and $T =G\sm (S\cup\{e\}).$ Since $\phi(|G|) \geq 2,$ choose $u,v \in G$ such that $|u|=|G|$ and $|v|=|G|$ with $u \neq v.$ Since both $u$ and $v$ are adjacent to every element in $H \setminus \{e\}$ in $\1$, $u$ and $v$ are not isolated vertices and so $u,v \in T$. Further $u$ and $v$ are not adjacent in $\1.$ This implies that $\langle T\rangle$ is not a complete subgraph in $\1$. These imply that $\1$ is not a split graph, which is a contradiction. 

Hence we have $H = G$.  Suppose  $\phi(p_i)\geq 2$ for $i=1,2$. Then one can choose distinct elements $x,y,u,v \in H$ such that $|x| =|y|=p_1$ and $|u| =|v|=p_2.$ Then $\langle\{x,y,u,v \}\rangle=2K_2$ in $\1$ and  $\1$ is connected. By Theorem~\ref{thm4}, $\1$ is not a split graph. Hence, we have $p_1=2, p_2=3$ and so $|G|=6.$ From this, $G\cong \mathbb{Z}_6$.
\end{proof}

A graph $\G$ is a {\it claw-free} graph if $\G$ does not contain a $K_{1,3}$ as an induced subgraph. In the following theorem, we characterize when $\1$ is a claw-free graph.

\begin{theorem} \label{10}
Let \(G\) be a finite cyclic group of order $n\geq 3$ and \(H\neq \{e\}\) be a subgroup of $G.$ Then the generalized non-coprime graph $\1$ is claw-free if and only if $H$ satisfies any one of the following.		
\begin{enumerate}
\item[\rm (i)] $H=G$ and $|G|\leq p_1^{\a_1}p_2^{\a_2}$ for two distinct primes $p_1$ and $p_2$ and $\a_1$ and $\a_2$ are non-negative integers;
\item[\rm (ii)] $ H \neq G $ and $G$ is isomorphic either $\mathbb Z_4$ or $\mathbb Z_6$.
\end{enumerate}		 
\end{theorem}

\begin{proof}
Suppose (i) is true. Let $S=\{u,v,x,y\}\se G \sm \{e\}.$ Due to the possibilities for orders of elements in $S,$ either $C_3$ is a subgraph of $\langle S\rangle$ or $\langle S\rangle=2K_2$ in $\1$ and hence $\1$ is claw-free.
 
Assume that (ii) is true. Then $H\neq G.$ If $G\cong \mathbb Z_4,$ then $H\cong \mathbb Z_2$ and so $\1=P_3.$ If $G\cong \mathbb Z_6$ and $H=\mathbb Z_2,$ then $\1=2K_1 \cup P_3.$ If $G\cong \mathbb Z_6$ and $H=\mathbb Z_3,$ then $\1$ is a disconnected graph containing contains at least two vertices of degree 3. In all these cases, $\1$ is claw-free.

Conversely assume that $\1$ is claw-free. Then $|G|$ cannot be a prime number when $H\neq G.$ Let us take $|G|=p_1^{\a_1}p_2^{\a_2} \ldots p_r^{\a_r}$ where $p_i$ are distinct primes and $\a_i$ are positive integers for $1\leq i\leq r.$ To complete the proof, it is enough to show that $\1$ is not claw-free when one of 
either (1) $H=G$ with $r \geq 3$, or (2) $H \neq G $ with $|G| \geq 8$.

(1) Assume that $H=G$ with $r \geq 3.$ Choose $x,y,z \in G$ such that $|x|=p_1$,
$|y|=p_2$, and $|z|=p_3.$ If $w \in G$ such that $|w|=|G|$, then
$\langle\{w,x,y,z \}\rangle=K_{1,3}\subseteq \1$ and so $\1$ is not claw-free.

(2) Assume that $  H \neq G $ with $|G| \geq 8.$ Then $\phi(|G|) \geq 4$. Choose $x,y,z \in G$ such that $|x|=|y|=|z|=|G|$ and $w \in H\sm \{ e\}$. Then $\langle\{w,x,y,z \}\rangle=K_{1,3}\subseteq  \1$ and so $\1$ is not claw-free.	
\end{proof} 

A graph $G$ is said to be a {\it chordal graph} if every cycle of length four or more  in $G$ contains a  chord. In the following theorem, we obtain a characterization for the graph $\1$ to be chordal.

\begin{theorem} \label{12} Let \(G\) be a finite cyclic group of order $\geq 5$ and \(H\neq \{e\}\) be a subgroup of $G$. Then the generalized non-coprime graph $\1$ is chordal if and only if either $H$ is a $p$-group or $H$ is not a $p$-group, $H=G$, and $|G| \leq  p_1^{\a_1}p_2^{\a_2}p_3^{\a_3}$.
\end{theorem}
\begin{proof} Assume that $H$ is a $p$-group.

If $H = G,$ then by Lemma~\ref{11.1}(iii), $\1$ is a complete graph and so $\1$ is chordal. If $H \neq G$ and $H \cong \mathbb Z_2,$ then $\1$ is star graph and so $\1$ is chordal. 

If $H \neq G$ and $H \ncong \mathbb Z_2,$ then $\langle H\setminus \{e\}\rangle$  is a complete subgraph. Suppose $S \subset G,$ $|S|\geq 4$ and $\langle S\rangle$ is a cycle in $\1$. Since $\langle G \setminus H\rangle$ is a totally disconnected subgraph of $\1$, $|S \cap (H \setminus \{e\}) |=\frac{|H|-1}{2}\geq2$. 

If $|S \cap (H \setminus \{e\}) |=2$, then $x,y \in S \cap (H \setminus \{e\})$ and $x\neq y$. Since $H$ is a $p$-group, $x,y$ are adjacent in $\1$. Since $\langle S\rangle$ is a cycle, there exists $z \in S \cap (G\setminus \{e,x,y\})$ such that $x$ is adjacent to $z$ and so $(|x|,|z|)\geq p$. Note that $(|x|,|y|)\geq p$ as $H$ is a $p$-group. This implies that $y$ is also adjacent to $z$ in $\1$ and so $\langle x,y,z\rangle=C_3$ in $\langle S\rangle.$ Hence $\1$ is a chordal graph. 

If $|S \cap (H \setminus \{e\}) |\geq 3$, then $\langle S\rangle$ contains $C_3$ as a subgraph of $\1$ by $\langle H \setminus \{e\}\rangle$ is a complete subgraph and so $\1$ is a chordal graph. 
	
Assume that $H$ is not a $p$-group, $H=G$, and $|G|\leq  p_1^{\a_1}p_2^{\a_2}p_3^{\a_3}.$ Suppose $S \s G\sm \{e\}, |S|\geq 4$ and  $\langle S\rangle$ is a cycle in $\1.$ Let $x \in G$. Then one of the following is true:

(i) $\theta_x=\{p_1 \}$; (ii) $\theta_x=\{p_2 \}$; (iii) $\theta_x=\{p_3 \}$; (iv) $\theta_x=\{ p_1,p_2\}$; (v) $\theta_x=\{p_1,p_3 \}$; (vi) $\theta_x=\{p_2,p_3 \}$; and (vii) $\theta_x=\{p_1,p_2,p_3 \}$. 

If $x \in S$ such that $\theta_x=\{p_1,p_2,p_3 \},$ then $x$ is adjacent to all other vertices in $\1$ and so $\deg(x) = |G|-2 \geq 3 $ in $\langle S\rangle$ which is a contradiction to $\langle S\rangle$ being a cycle. Thus $x\notin S$. If $x \neq y \in S$ with $\theta_x=\theta_y$, then $C_3$ is a subgraph of $\langle S\rangle$. This turns that $S \subseteq \{u,v,w,x,y,z \} \subset G$ such that  $|u|=p_1$, $|v|=p_2$, $|w|=p_3$, $|x|=p_1p_2$, $|y|=p_1p_3$ and $|z|=p_2p_3$. Since $\langle \{x,y,z\}\rangle=C_3,$ we have $|S \cap \{x,y,z\}|\leq 2$ and so $\langle S\rangle \neq C_6$ in $\1$. Since $\langle \{u,v,w\}\langle=\overline{K}_3$, we have $|S \cap \{u,v,w\}|\leq 2$. This implies that $\langle S\rangle=P_4$ if $|S|=4$, which is a contradiction to $\langle S\rangle=C_{|S|}$ with $|S|\geq 4$. These facts imply that for $x,y\in S$, $\theta_x$ contains at most two primes and $\theta_x\neq \theta_y$ for $x\neq y$. Hence $\1$ is a chordal graph.

Conversely assume that $\1$ is a chordal graph. Suppose that $H=G$ and $r \geq 4$, where $r$ is the number of prime divisors of $|G|$. Choose distinct elements $u,v,x,y \in G$ such that $|u|=p_1p_2$, $|v|=p_2p_3$, $|x|= p_3p_4 $ and $|y|=p_1p_4$. Then $\langle\{u,v,x,y \}\rangle=C_4$ in $\1$, which is a contradiction. Hence $r\leq 3$ whenever $H=G$.

Suppose $H \neq G $ and $H$ is not a $p$-group. Then there exist at least two distinct primes $p_1$ and $p_2$ dividing $|H|$. Trivially $p_i\mid |G|$ for $i=1,2$. Hence $\phi(|G|) \geq 2$. Choose distinct elements $u,v,x,y \in G$ such that $|u|=p_1$, $|v|=p_2$ and $|x|=|y|=|G|$. Then $\langle\{u,v,x,y \}\rangle=C_4$ in $\1,$ which is a contradiction.  Hence $H$ is a $p$-group.

Hence $\1$ is chordal if and only if either $H$ is a $p$-group or $H = G$ with $r \leq 3$. 
\end{proof}

\subsection{Perfect characterization for the non-coprime graph}
A graph $\Gamma$ is perfect if the clique number and chromatic number of
every induced subgraph of $\Gamma$ are equal to one another. By the
\emph{Strong Perfect Graph Theorem}~\cite{CRST}, a graph is perfect if
and only if it does not contain an odd cycle of length greater than~$3$
or the complement of one as an induced subgraph. We also note the earlier
\emph{Weak Perfect Graph Theorem} of Lov\'asz~\cite{lovasz}, which asserts
that the complement of a perfect graph is perfect.

In this section, we will determine which generalized non-coprime graphs
of cyclic groups are perfect.

The proof is not straightforward, and we need a couple of tools.

The most important tool is \emph{twin reduction}, described in detail in
\cite[Section 7]{c:survey}. Two vertices $v,w$ of a graph $\Gamma$ are
\emph{twins} if they have the same neighbours except possibly for one
another. (We call $v$ and $w$ open twins if their open neighbourhoods are
equal, and closed twins if their closed neighbourhoods are equal.) The
process of twin reduction is the following. Choose a pair of twins, and
identify them (equivalently, delete one); repeat. It is shown in the
cited reference that if twin reduction is carried out until no further
twins remain, the resulting graph is (up to isomorphism) independent of
the way the twin reduction was carried out.

The important fact for us is the following obvious proposition.

\begin{proposition}
Let $\Gamma$ and $\Delta$ be graphs, such that $\Delta$ contains no pairs
of twin vertices, and let $\Gamma'$ be obtained from $\Gamma$ by a sequence
of twin reductions. Then $\Delta$ is an induced subgraph of $\Gamma'$ if
and only if it is an induced subgraph of $\Gamma$.
\end{proposition}

The relevance for us is that cycles of length greater than $4$, and
their complements, contain no pairs of twins; so in deciding whether a
graph is perfect, we may apply twin reduction without changing the result.
(Cycles of lengths~$3$ and $4$ contain twin vertices, and so our method
cannot detect these.)

The first application of twin reduction to generalized non-coprime graphs 
$\Gamma_{G,H}$ of cyclic groups comes from the fact that two vertices of
such a graph which have the same order are twins. This is because their orders
are not coprime (as we have excluded the identity) and if one of them is in
$H$ then so is the other.

There is a more specific reduction which will also be useful. Let $\Gamma$
be a graph, and $v$ a vertex of $\Gamma$. If $v$ satisfies any of the
following conditions, then $v$ cannot be contained in an induced subgraph
which is an odd cycle or the complement of one:
\begin{itemize}
\item $\deg(v)<2$ or $|V(\Gamma)-1-\deg(v)|<2$;
\item the neighbourhood of $v$ is a complete graph, or the non-neighbourhood
of $v$ is a null graph.
\end{itemize}
This is clear for the first condition, since a vertex of a cycle of length
at least~$5$ has degree and codegree at least $2$. For the second condition,
note that, in an induced cycle of length at least~$5$, the neighbours of a
vertex are nonadjacent, while the set of non-neighbours contains an edge.

\begin{theorem}
Let $G$ be a cyclic group of order $p_1p_2p_3p_4$, where $p_1,\ldots,p_4$
are distinct primes, and let $H\ne\{e\}$ be a subgroup of $G$. Then the
generalized non-coprime graph $\Gamma_{G,H}$ is perfect. Moreover, in all
cases except when $|H|$ is the product of three primes, $\Gamma_{G,H}$ 
contains no cycles of length greater than~$4$ or complements of such cycles.
\label{t:4primes}
\end{theorem}
\begin{proof}
As noted above, two vertices of $\Gamma_{G,H}$ which have the same order are
twins, so we may begin by collapsing all elements of each given order to a
single vertex. This graph $X$ can be regarded as an induced subgraph of
$\Gamma_{G,H}$.

There are fifteen possible orders of elements in $G$, namely the fifteen
products of non-empty subsets of $\{1,2,3,4\}$. We denote the collapsed graph
by $X$. The vertices $x_I$ of $X$ are labelled by non-empty subsets $I$ of
$\{1,2,3,4\}$. We ease notation by writing, for example, $x_{12}$ for
$x_{\{1.2\}}$.
The intersection $X\cap H$ is $\{X_J:J\subseteq L\}$, where $L$ is the set
$\{i:p_i\mid|H|\}$. Two vertices $x_J$ and $x_K$ are joined if and only if
$J\cap K\ne\emptyset$ and at least one of $I$ and $J$ is contained in $L$.
Moreover, there is complete symmetry among $1,2,3,4$.
So, up to isomorphism, there are only four possibilities for the induced
subgraph on $X$, corresponding to $L=\{1,\ldots,i\}$ for $i=1,\ldots,4$.

To prove the theorem, we have to show that these graphs or their complements
contain no induced cycle of length greater than or equal to $5$, except in the case
where $|L|=3$ (when there is a $6$-cycle). We consider the four cases.

{\bf Case 1.} $|L|=1,$ say $L=\{1\}$. Then the graph is a star $K_{1,7}$
together with seven isolated vertices, and the result is clear. Alternatively,
continuing with twin reduction reduces this graph to a single vertex. (We
collapse the seven leaves, and the seven isolated vertex, giving an edge and
an isolated vertex. Then we collapse the two ends of the edge, and finally the
remaining two vertices.)

{\bf Case 2.} $|L|=2,$ say $L=\{1,2\}$. The three vertices $x_I$ with
$I\subseteq L$ induce a path. The remaining $12$ vertices fall into four
sets of three depending on the intersection $I\cap L$. Vertices within the
same set of three are twins. Collapsing them gives the graph shown in 
Figure~\ref{f:f1}.

\begin{figure}[htbp]
\begin{center}
\setlength{\unitlength}{2mm}
\begin{picture}(24,18)
\multiput(0,0)(10,0){3}{\circle*{1}}
\multiput(0,8)(10,0){3}{\circle*{1}}
\put(10,16){\circle*{1}}
\put(0,8){\line(1,0){20}}
\put(0,0){\line(0,1){8}}
\multiput(0,0)(0,8){2}{\line(5,4){10}}
\multiput(20,0)(0,8){2}{\line(-5,4){10}}
\put(20,0){\line(0,1){8}}
\put(10,8){\line(0,1){8}}
\put(-2,9){$\scriptstyle{x_{1}}$}
\put(11,9){$\scriptstyle{x_{12}}$}
\put(21,9){$\scriptstyle{x_{2}}$}
\put(11,17){$\scriptstyle{x_{123}}$}
\put(-3,0){$\scriptstyle{x_{13}}$}
\put(11,0){$\scriptstyle{x_{3}}$}
\put(21,0){$\scriptstyle{x_{23}}$}
\end{picture}
\end{center}
\caption{\label{f:f1}A reduced graph}
\end{figure}
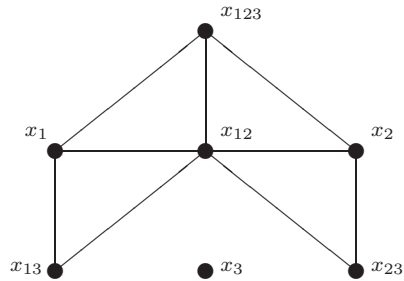

We can delete the isloated vertex, and then the central vertex (which is
connected to all of the remaining vertices), leaving a $5$-vertex path.
So there are no cycles of length greater than~$4$, and the graph is perfect.

{\bf Case 3.} $|L|=3,$ say $L=\{1,2,3\}$. In this case the graph has no
twins, but can still be pruned.

Vertex $x_4$ is isolated, and can go; then $x_{123}$ is joined to all
others and can go. Next, $x_{14}$ has neighbourhood which is complete, and
so can go, and similarly $x_{24}$ and $x_{34}$. Then $x_{1234}$ has
non-neighbourhood which is null, and so can go. Finally, $x_{12}$ is joined
to all vertices except $x_3$, and can go; similarly $x_{13}$ and $x_{23}$.
The remaining six vertices form a $6$-cycle.

Returning to the graph $\Gamma_{G,H}$, we see that, if $G$ is cyclic of order
$p_1p_2p_3p_4$, where $p_1,\ldots,p_4$ are distinct primes, and $H$ is the
subgroup of order $p_1p_2p_3$, then elements of $G$ with orders
\[p_1,p_1p_2p_4,p_2,p_2p_3p_4,p_3,p_1p_3p_4\]
form a $6$-cycle.

{\bf Case 4.} $L=\{1,2,3,4\}.$ Then $x_I$ and $x_J$ are joined if and only
if $I\cap J\ne\emptyset$. So, if $|I|\ge3$, then $x_I$ is nonadjacent to at
most one vertex, and can be deleted. Also, if $|I|=1$, then the neighbourhood
of $x_I$ is a complete graph, so these vertices can also be deleted. The 
remaining six vertices (with $|I|=2$) induce a complete tripartite graph
$K_{2,2,2}$, which contains no cycle of length greater than~$4$.
\end{proof}

By using the arguments used in the proof of Theorem~\ref{t:4primes}, we have the following theorem.
\begin{theorem} \label{9.03} Let \(G\) be a finite cyclic group and \(H\neq \{e\}\) be a subgroup of $G.$ Then length of any induced odd cycle is at most $3$ in both the generalized non-coprime graph $\1$ and its complement $\c$ if $H$ satisfies any one of the following:
\begin{enumerate}		 
\item[\rm	(a)] $|H|\leq|G|\leq p_1^{\a_1}p_2^{\a_2}p_3^{\a_3};$
\item[\rm	(b)] $|G|=p_1^{\a_1}p_2^{\a_2}\cdots p_r^{\a_r}$ with $r \geq 5$ and $|H|\leq p_i^{\a_i}p_j^{\a_j}p_k^{\a_k}$ with $i, j,k$ are distinct. 
\end{enumerate}		 
\end{theorem}

\begin{theorem} \label{9.04}
Let \(G\) be a finite cyclic group and \(H\) be a non-trivial subgroup of $G$. 
Then the generalized non-coprime graph  $\1$ is a perfect graph if and only if $H$ satisfies any one of the following:
\begin{enumerate}		 
\item[\rm	(a)] $|H|=|G|=p_1p_2 p_3 p_4;$ 
\item[\rm	(b)] $|H|\leq|G|= p_1^{\a_1}p_2^{\a_2}p_3^{\a_3};$
\item[\rm	(c)] $|G|=p_1^{\a_1}p_2^{\a_2}\cdots p_r^{\a_r}$ with $r \geq 5$ and $|H|\leq p_i^{\a_i}p_j^{\a_j}p_k^{\a_k}$ with $i, j,k$ are distinct. 
\end{enumerate}
\end{theorem}
\begin{proof}
Assume that $\1$ is a perfect graph. Let $h_r$ be the number of distinct prime divisors of $|H|.$
To complete the proof, it is enough to prove that $\1$ contains an induced cycle of odd order with length greater than or equal to 5 in the following cases:
 
(1) $r \geq 5$ and $h_r \geq 5;$

(2) $r \geq 5$ and $h_r=4;$
 
(3) $H\neq G$, $r = 4 \t{~with~} \a_j \geq 2 $ for some $j$ and $h_r=4.$

(1) Suppose $h_r \geq 5.$ Without loss of generality, one can assume that $p_i \mid |H|$ for $1 \leq i \leq 5.$ Let $S=\{u,v,w,x,y \}\s H$ such that $|u|=p_1p_2,$ $|v|=p_1p_3,$ $|w|=p_3p_4,$ $|x|=p_4p_5$ and $|y|=p_2p_5.$ Then $\langle S\rangle=C_5$ in $\1.$ 

(2) Suppose $r \geq 5$ and $h_r = 4.$ Without loss of generality, one can assume that $p_i \mid |H|$ for $1 \leq i \leq 4.$ Let $S=\{u,v,w,x,y \}\s G$ such that $|u|=p_1,$ $|v|=p_2p_4,$ $|w|=p_3p_4,$ $|x|=p_1p_2p_5$ and $|y|=p_1p_3p_5.$ Then $\langle S\rangle=C_5$ in $\1.$ 

(3) Suppose $r = 4 \t{~and~} \a_j \geq 2 $ for some $j$ and $h_r=4.$ Since $H \neq G,$ we have $\beta_i<\a_i$ for some $i.$ Without loss of generality, assume that $p_i=p_1$ and $\beta_1 < \a_1 .$ Let $S=\{u,v,w,x,y \}\s G$ such that $|u|=p_1,$ $|v|=p_2p_4,$ $|w|=p_3p_4,$ $|x|=p_2p_1^{\a_1}$ and $|y|=p_3p_1^{\a_1}.$ Then $\langle S\rangle=C_5$ in $\1.$ 

The converse follows from Theorems~\ref{t:4primes} and \ref{9.03}.	
\end{proof}

\section{Nilpotent groups}
In this section, we will see that extending our class of groups from cyclic to
nilpotent groups (which predominate in the enumeration of finite groups) gives
no further generalized non-coprime graphs. So all the detailed results about
generalized non-coprime graphs of cyclic groups extend to nilpotent groups.

\begin{theorem}
Let $G$ be a nilpotent group of order $n$, and $H$ a subgroup of $G$ of
order~$m$. Then $\Gamma_{G,H}\cong\Gamma_{\mathbb{Z}_n,\mathbb{Z}_m}$.
\end{theorem}

The difficulty in the theorem is in keeping track of the subgroup $H$; we do
not know how to recognise the vertices in $H$ in the graph $\Gamma_{G,H}$ for
arbitrary $G$ and $H$. For this purpose, we introduce a new gadget $\tc(G,H)$,
the \emph{tagged coprime graph} of $G$ and $H$. This consists of a graph
$\Gamma$ and a distinguished subset $S$ of vertices, and is constructed as
follows:
\begin{itemize}
\item The vertex set of the graph is $G$.
\item We join vertices $x$ and $y$ if their orders are coprime (this includes
putting a loop at the identity vertex).
\item The tagged set is $H$.
\end{itemize}

Now the proof of the theorem consists of a number of lemmas.

\begin{lemma}
We can recover the generalized non-coprime graph $\Gamma_{G,H}$ from \linebreak $\tc(G,H)$ by the following procedure:
\begin{itemize}
\item Delete the identity.
\item Add each edge joining a pair of vertices which are not tagged, if this
edge is not already present.
\item Ignore the tags.
\item Take the complementary graph.
\end{itemize}
\label{l1}
\end{lemma}

This is clear.

\begin{lemma}
Let $|G|=p^n$ and $|H|=p^m$ with $m\le n$. Then $\tc(G,H)$ consists of a star
on $1+(p^n-1)$ vertices, the centre of which carries
a loop, with $p^m$ tagged vertices including the loop.
\label{l2}
\end{lemma}

For the coprime pairs form a star whose centre is the identity; the tagged
vertices include the identity. Note that the structure of this tagged graph
depends only on $|G|$ and $|H|$.

\medskip

The \emph{categorical product} $\Gamma\times\Delta$ of two graphs $\Gamma$
and $\Delta$ has vertex set the cartesian product of the vertex sets of
$\Gamma$ and $\Delta$, with $(v,x)$ joined to $(w,y)$ if and only if 
$v$ is joined to $w$ in $\Gamma$ and $x$ is joined to $y$ in $\Delta$.

Now we define the product of tagged coprime graphs $\tc(G_1,H_1)$ and
$\tc(G_2,H_2)$ to be the pair $(G_1\times G_2,H_1\times H_2)$, where the
first $\times$ is the categorical product of graphs and the second is the
Cartesian product of sets.

\begin{lemma}
Let $G_1$ and $G_2$ be groups with $\gcd(|G_1|,|G_2|)=1$, and $H_i$ a
subgroup of $G_i$ for $i=1,2$. Then
\[\tc(G_1\times G_2,H_1\times H_2)=\tc(G_1,H_1)\times\tc(G_2,H_2).\]
\label{l3}
\end{lemma}

\begin{proof}
The graph on the right has the same vertex set and the same set of
tagged vertices as the graph on the left. We have to check edges.

Suppose that $\{g,g'\}$ is an edge of $\tc(G_1,H_1)$ and $\{x,x'\}$ an
edge of $\tc(G_2,H_2)$. Now $\gcd(|g|,|x|)=1$ since $g$ and $x$ belong to
groups of coprime order. Similarly $\gcd(|g'|,|x'|)=1$. Also, by 
definition of the tagged coprime graph, $\gcd(|g|,|g'|)=\gcd(|x|,|x'|)=1$.
Hence $\gcd(|(g,x)|,|(g',x')|)=1$, and so $(g,x)$ and $(g',x')$ are joined in
$\tc(G_1,H_1)\times\tc(G_2,H_2)$. The converse is proved similarly. Thus
the graph in the pair $\tc(G_1,H_1)\times\tc(G_2,H_2)$ is the categorical
product of those in $\tc(G_1,H_1)$ and $\tc(G_2,H_2)$.
\end{proof}

This result extends to products of any number of groups of pairwise
coprime order.

\medskip

Now we can prove the theorem. Let $G$ be a nilpotent group. Then $G$ is the
direct product of its Sylow subgroups $G_p$ (as $p$ runs over the prime
divisors of $|G|$); these groups have pairwise coprime order.
Any subgroup $H$ of $G$ is also nilpotent,  so is the direct product of its
intersections with the Sylow subgroups of $G$.

Now $\Gamma_{G,H}$ consists of a complete graph of order $p^n$ with
the edges of  a complete subgraph of order $p^m$ deleted, and then isolated
vertices removed. Clearly this is isomorphic 
to $\Gamma_{\mathbb{Z}_{p^n},\mathbb{Z}_{p^m}}$.

Next, $\tc(G_p,G_p\cap H)$ is clearly determined by $\Gamma_{G_p,G_p\cap H}$
for each prime $p$; then Lemma~\ref{l3} determines $\tc(G,H)$, and then
Lemma~\ref{l1} determines $\Gamma_{G,H}$. Since also
\[\mathbb{Z}_{|G|}=\prod \mathbb{Z}_{p_i^{n_i}}\]
where $|G|=\prod p_i^{n_i}$, we see that the same process determines
$\Gamma_{\mathbb{Z}_{|G|},\mathbb{Z}_{|H|}}$. The theorem is proved.

\begin{problem} Is it possible to identify, up to graph isomorphism, the
vertices corresponding to elements of $H$ in the generalized non-coprime
graph $\Gamma_{G,H}$?
\end{problem}

\begin{remark}
 We note that the theorem does not extend beyond nilpotent
groups. The smallest non-nilpotent group is the symmetric group $S_3$, whose
generalized non-coprime graphs are different from those of $\mathbb{Z}_6$.
\end{remark}

\section{EPPO groups}
At the other extreme from nilpotent groups are the \emph{EPPO groups}
(sometimes called CP-groups) in which every element has prime power order.
After pioneering work by Higman in the 1950s and Suzuki in the 1960s, a
complete description was given by Brandl~\cite{brandl} in 1981; a modern
and detailed account can be found in the survey paper \cite{cm}.

Let $\Omega_p(G)$ be the set of non-identity elements of $p$-power order
in $G$,  where $p$ is a prime divisor of $|G|$. If $G$ is an EPPO group,
then the non-coprime graph $\Gamma_G$ is the disjoint union of the subgraphs
on the sets $\Omega_p(G)$. If we let $X(n,m)$ denote the complete graph
on $n$ vertices with the edges of a complete graph on $m$ vertices removed,
then the following is true:

\begin{theorem}
Let $G$ be an EPPO group, and $H$ a subgroup of $G$. Then $\Gamma_{G,H}$
is the disjoint union of graphs $X(n_p,m_p)$, with $n_p=|\Omega_p(G)|$ and
$m_p=|\Omega_p(G)\setminus\Omega_p(H)|$, where $p$ runs over the prime
divisors of $|G|$.
\end{theorem}

As a corollary, we see that $\Gamma_{G,H}$ is connected (apart from isolated
vertices) if and only if $H$ contains $\Omega_p(G)$ for all but one
prime divisor of $|G|$.

\section{The Gruenberg--Kegel graph}

We conclude with a connection between the non-coprime graph of a group $G$ and
a smaller and well-studied graph, the \emph{Gruenberg--Kegel graph}, or GK
graph for short (sometimes called the prime graph). This was introduced by
Gruenberg and Kegel in their study of the integral group ring of a finite
group. The vertices of the GK graph are the prime divisors of $|G|$; there
is an edge joining primes $p$ and $q$ if $G$ contains an element of order $pq$.

Gruenberg and Kegel showed that connectedness of this graph is equivalent to
indecomposability of the augmentation ideal of the group ring of $G$. They
proved a structure theorem for groups whose GK graph is disconnected. They
did not publish this result, which appeared first in a paper by Gruenberg's
student Williams~\cite{williams}, and was refined by subsequent authors. For
further discussion see~\cite{cm}.

The GK graph shows clearly the contrast between nilpotent and EPPO groups. The
GK graph of any nilpotent group is complete; but a group is EPPO if and only if
its GK graph is null.

The connection with our current topic is the following.

\begin{theorem}
For a finite group $G$, the non-coprime graph $\Gamma_G$ is connected if and
only if the Gruenberg--Kegel graph is connected.
\end{theorem}

\begin{proof}
Suppose first that the GK graph is connected. Given any two prime divisors of
$|G|$, say $p$ and $r$, there is a path joining them in the GK-graph. This
lifts to a path of twice the length between elements of orders $p$ and $r$
in $\Gamma_G$, where we replace the edge $\{p,q\}$ by elements of orders
$p$, $pq$ and $q$. Hence the elements of prime order are contained in a single
connected component of $\Gamma_G$. But any non-identity element of $G$ is
joined to an element of prime order in $\Gamma_G$, namely some power of itself.

Conversely, suppose that $\Gamma_G$ is connected. Let $g$ and $h$ be elements
of prime orders $p$ and $r$ in $G$, and choose a path joining them in
$\Gamma_G$, say $(g_0,g_1,\ldots,g_d)$. Let $i$ be minimal such that $p$ does
not divide $|g_i|$, and let $q$ be a prime dividing $|g_{i-1}|$ and
$|g_i|$. Then some power of $g_{i-1}$ has order $pq$, so $p$ and $q$ are
joined in the GK graph. Continuing this process, we find a path from $p$ to $r$
in the GK graph.
\end{proof}

In this connection we mention a very similar result, given in
\cite[Theorem 9.2]{c:survey}. The \emph{commuting graph} of a finite group
has the non-central elements as vertices, two vertices joined if they commute.

\begin{theorem}
Let $G$ be a group whose centre contains just the identity. Then the
commuting graph of $G$ is connected if and only if its Gruenberg--Kegel graph
is connected.
\end{theorem}

\begin{corollary}
Let $G$ be a group whose centre contains just the identity. Then the
non-coprime graph of $G$ is connected if and only if the commuting graph is
connected.
\end{corollary}

We have not been able to extend this result to the generalized non-coprime
graph.

\begin{problem}
Is there a ``relative GK graph'' defined by a group $G$ and subgroup $H$, whose
connectedness is equivalent to that of the generalized non-coprime graph
$\Gamma_{G,H}$?
\end{problem}

\section{Further comments and open problems}
We have generalized the non-coprime graph of a finite group $G$ by
introucing a subgroup $H$ as an extra parameter, where we keep only those
edges of the non-coprime graph which have at least one vertex in $H$.
This suggests several questions:

\begin{problem}
Take another graph defined on a group, for example the power graph,
commuting graph or non-generating graph, and generalize it as above by
introducing a subgroup $H$. What can be said about such graphs?
\end{problem}

\begin{problem}
Given $\Gamma_{G,H}$ as a graph, is it possible to identify the set of
vertices comprising the subgroup $H$ of $G$, at least up to automorphisms
of the graph?
\end{problem}

\section*{Conflict of interest}
 The authors declare that they have no conflict of interest.
\section*{Data Availability Statement }
 The authors have not used any data for the preparation of this manuscript.

\end{document}